\documentclass[11pt]{amsart}
\usepackage{amsmath,amsfonts,amsthm,amscd,amssymb,graphicx}

\newtheorem{theorem}{Theorem}[section]

\newtheorem{lemma}[theorem]{Lemma}

\newtheorem{corollary}[theorem]{Corollary}

\newtheorem{remark}[theorem]{Remark}

\newenvironment{acknowledgements}[1][Acknowledgements:]{\begin{trivlist}
\item[\hskip \labelsep {\bfseries #1}]}{\end{trivlist}}

\begin{document}

\begin{abstract}
Let $B$ be a finite, separable von Neumann algebra. We prove that a $B$-valued 
distribution $\mu$ that is the weak limit of an infinitesimal array is infinitely divisible.  The proof of this theorem utilizes the 
Steinitz lemma and may be adapted to provide a nonstandard proof of this type of theorem for various other probabilistic categories.  We also develop weak topologies
for this theory and prove the corresponding compactness and convergence results.
\end{abstract}
\address{Texas A\&M University, Dept. of Mathematics, 
\\ Mail Stop 3368 , College Station, TX 77843-3368 
\\jwilliams@math.tamu.edu}
\title[Operator-Valued Infinite Divisibility.]{An Analogue of Hin\u{c}in's Characterization of Infinite Divisibility for Operator-Valued Free Probability.}
\author{John D. Williams}
\maketitle

\section{Introduction.}
As a starting point, we recall a theorem, due to Hin\u{c}in in classical probability theory.
Let $\mathcal{M}$ denote the collection of all probability measures on $\mathbb{R}$.  Given $\mu, \nu \in \mathcal{M}$, 
we denote by $\mu \ast \nu$ the distribution of $X + Y$ where $X$ and $Y$ are independent random variables with
distribution $\mu$ and $\nu$ respectively.  For $r\in \mathbb{R}$ we denote by $\delta_{r}\in \mathcal{M}$ the 
Dirac mass at $r$.
\begin{theorem}\label{classical_H}
Let $\{\mu_{ij} \}_{i\in\mathbb{N},j=1,\ldots,n_{i}}$ and  $\mu$ be probility measures on $\mathbb{R}$.
Assume that the following properties hold:
\begin{enumerate}
\item
For every $\epsilon >0$, we have that $\lim_{i\uparrow \infty} \mu_{ij}([-\epsilon,\epsilon])\rightarrow 1$.
\item
There exists a sequence $\{r_{i} \}_{i\in \mathbb{N}} \subset \mathbb{R}$ such that 
$\delta_{r_{i}}\ast \mu_{i1} \ast \cdots \ast \mu_{in_{1}} \rightarrow \mu$ weakly.
\end{enumerate}
Then, for every $N\in \mathbb{N}$, there exists a probability measure $\mu_{1/N}$ such 
that $\mu = \mu_{1/N} \ast \cdots \ast \mu_{1/N}$, where the convolution on the right hand side is $N$-fold.
\end{theorem}The conclusion of the above Theorem is the assertion that $\mu$ is \textit{infinitely divisible} with respect to 
the convolution operation.

Free probability was developed in the 80's by Voiculescu as a method for encoding free product
phenomenon in operator algebras in a probabalistic setting.  In \cite{vo2}, Voiculescu introduced
\textit{free indepence} which is a noncommutative analogue of classical independence.  The corresponding 
convolution operation, \textit{free additive convolution} (in symbols, $\boxplus$) was also introduced.  In particular,
if $X$ and $Y$ are freely independent random variables with respective distributions $\mu$ and $\nu$, then we denote
by $\mu \boxplus \nu$ the distribution of the random variable $X + Y$.

In this noncommutative setting, analogues of Theorems \ref{classical_H} have been developed.  Indeed, it was shown by Bercovici
and Pata in \cite{Pata} that the same result is true if classical independence is replaced by free independence and classical
infinite divisibility is replace by $\boxplus$-infinite divisibility.  Belinschi and Bercovici proved in \cite{Bel} that this
characterization of infinite divisibility held for \textit{multiplicative free convolution}, which arises when taking
the product of free random variables.  The primary focus of this paper will be an operator valued generalization of this theorem
for additive convolution.

In order to address amalgamated free product phenomenon in operator algebras with probabilistic methods, operator valued 
versions of free probability theory were developed in \cite{vop}, \cite{vo5} and \cite{vo6}.The purpose of
this paper is to prove Theorem \ref{mainresult}, which is a version of Theorem \ref{classical_H} for operator
valued free probability.  A weaker version of this type of theorem was proven in \cite{PVB} as a necessary intermediate lemma
in their proof of an operator valued generalization of the Bercovici-Pata bijection.

This paper is organized as follows.  Section $(2)$ contains preliminaries for $B$-valued distributions.
Section $(3)$ is devoted to the preliminaries of operator valued free probability.  Section $(4)$ contains
preliminary results related to those distributions that arise from tracial von Neumann algebras.  Section $(5)$ is devoted
to the Steinitz lemma which is the underlying tool in our approach to this theorem.  Section $(6)$ contains the main result
and section $(7)$ includes concluding remarks and acknowledgements.

\section{B-Valued Distributions.}

Let $B$ denote a unital $C^{\ast}$-algebra and $B\langle X \rangle$ the space of noncommutative polynomials over $B$.
  We say that a map $\mu: B\langle X \rangle \rightarrow B$ is \textit{completely positive} if for any finite set of
elements $P_{1}(X), \ldots, P_{n}(X) \in B\langle X \rangle$ we have that the matrix $$|\mu(P_{i}^{\ast}(X)P_{j}(X)) |_{i,j=1}^{n}$$ is a positive element of $M_{n}(B)$.
A map $\mu: B\langle X \rangle \rightarrow B$ is $B$-bimodular if $\mu(bP(X)b') = b\mu(P(X))b'$ for all $b,b' \in B$ and $P(X) \in B\langle X \rangle$.  We denote
by $\Sigma$ the space of all $B$-bimodular, completely positive, $B$ valued maps.

Let $\mu \in \Sigma$.  We say that $\mu$ is \textit{exponentially bounded} by $M$ if for all elements $b_{1}, \ldots, b_{n} \in B$, we have that
 $\| \mu(Xb_{1}X \cdots Xb_{n}X) \| \leq M^{n+1} \|b_{1} \| \cdots \|b_{n} \|$.  We denote by $\Sigma_{0}$ the set of all exponentially bounded elements
in $\Sigma$ and by $\Sigma_{0,M}$ the set of all such elements with a bound of $M$.  The elements $\mu \in \Sigma$ which arise as the distribution of elements of 
in a $C^{\ast}$-probability space are of primary interest in this paper.  We defer the development of the theory of these distributions to the next section.

We endow $\Sigma$ with the topology of pointwise weak convergence 
(that is $\mu_{\lambda} \rightarrow \mu$ if an only if $\mu_{\lambda} (P(X)) \rightarrow \mu(P(X))$
in the weak topology on $B$ for all $P(X) \in B\langle X \rangle$).  See \cite{PV} for a study of norm topologies on $\Sigma$.

\begin{lemma}\label{compact}
If we further assume that $B$ is a $W^{\ast}$-algebra, then the space $\Sigma_{0,M}$ is compact in the topology of pointwise weak convergence for all $M > 0$.
\end{lemma}

\begin{proof}
We follow the proof of Alaoglu's theorem.  We assume without loss of generality that $M=1$

Let $\{P_{i} \}_{i\in I} \subset B\langle X \rangle$ be a family of monomials with $B$-bilinear span equal to $B\langle X \rangle$ that possess the property
that each $P_{i}$ may be written as a product $P_{i} = Xb_{1}Xb_{2}\cdots Xb_{n}X$ with  $\| b_{i} \| \leq 1$ for $i=1,\ldots,n$.  
Let $K = \prod_{i\in I} B_{1}$, where $B_{1}$ denotes the unit ball in $B$, and endow this space $K$ with the product of the weak toplogies on $B$.  By 
Alaoglu and Tychonoff's theorems, K is compact.

Let $\Phi: \Sigma_{0,1} \rightarrow K$ by letting $\Phi(\mu) = (\mu(P_{i}))_{i\in I}$.  Our exponential bound condition implies that the image is indeed in $K$.
Since these elements $P_{i}$ have dense linear span,
this map is injective.  Observe that if $\mu_{\lambda} \rightarrow \mu$ in $\Sigma_{0,1}$ then for any $\sigma \subset I$ finite, we have that
$\mu_{\lambda}(P_{i}) \rightarrow \mu(P_{i})$ uniformly over $i\in \sigma$.  Thus, $\Phi$ is continuous.  Similarly, if $\Phi(\mu_{\lambda}) \rightarrow \Phi(\mu)$,
we have that $\mu_{\lambda}(P_{i}) \rightarrow \mu(P_{i})$ weakly so that $\Phi^{-1}$ is continuous.  Thus, $\Phi$ is a homeomorphism of $\Sigma_{0,1}$ onto its image.

Our lemma will follow when we show that $\Phi(\Sigma_{0,1})$ is a closed subset of $K$.  Assume that $\Phi(\mu_{\lambda}) \rightarrow (b_{i})_{i\in I}$.  We define a map
$\mu : B\langle X \rangle \rightarrow B$ by letting $\mu(P_{i}) = b_{i}$ and then extending this map by linearity.  Clearly, $\mu$ is well defined.  We claim that
$\mu \in \Sigma_{0,1}$.  To show this we must show that $\mu$ is $B$-bimodular, completely positive and has exponential bound equal to $1$.

First, observe that for $\phi \in B^{\ast}$ and $b,b' \in B$, we have 

 $$\phi(b\mu(P(X))b') - \phi(\mu(bP(X)b')) $$
$$ = \lim_{\lambda} (\phi(b\mu(P(X))b') - \phi(b\mu_{\lambda}(P(X))b') + (\phi(\mu_{\lambda}(bP(X)b')) -  \phi(\mu(bP(X)b'))$$
The right hand side of the equation is $0$ so that $\mu$ is $B$-bimodular.

Next, note that for a finite collection $\{f_{j} \}_{j=1}^{k} \subset B\langle X \rangle$, the matrix
$[\mu(f_{i}^{\ast}f_{j}) ]_{i,j=1}^{k}$ is the weak limit of $[\mu_{\lambda}(f_{i}^{\ast}f_{j}) ]_{i,j=1}^{k}$
with $M_{k}(B)$ endowed with the weak toplogy.  As the positive cone is weakly closed, the fact that $[\mu_{\lambda}(f_{i}^{\ast}f_{j}) ]_{i,j=1}^{k}$
is positive implies the same for $[\mu(f_{i}^{\ast}f_{j}) ]_{i,j=1}^{k}$.  Thus, $\mu$ is completely positive.

Lastly, for any monomial $P(X) = b_{1}Xb_{2}X\cdots b_{n}X$ in $B\langle X \rangle$, we have that $\mu_{\lambda}(P(X)) \leq \|b_{1} \| \cdots \|b_{n} \|$.
Alaoglu's theorem implies that $\mu(P(X))$ has the same bound.  Thus, our lemma holds.
\end{proof}

\section{Operator Valued Free Probability.}

Let $(A,\phi,B)$ be a triple with $B \subset A$ an inclusion of  $C^{\ast}$-algebras and $\phi: A \rightarrow B$ a B-bimodular, completely positive map.  
We shall refer to such a triple as
a \textit{$B$-valued probabliity space}.  Given an element $a \in A$ we denote by $\mu_{a}: B\langle X \rangle \rightarrow B$
 the \textit{$B$-valued distribution} of $a$, defined by the equation $\mu_{a}(P(X)) = \phi(P(a))$ for all $P(X) \in B\langle X \rangle$.  


We say that a family of subalgebras $\{A_{i} \}_{i\in I}$ are \textit{$B$-freely independent} if 
$\phi(a_{1}\cdots a_{n}) = 0$ whenever $a_{j} \in A_{i_{j}}$ satisfies $i_{j} \neq i_{j+1}$ for all $j=1,\ldots,n-1$
and $\phi(a_{\ell})=0$ for all $\ell = 1,\ldots,n$.  Given elements $\rho, \nu \in \Sigma_{0}$ that arise as the $B$-valued distributions of elements $x \in A$
and $y \in A'$, there exists a larger $C^{\ast}$-algebra $A \ast_{B} A'$ the contains copies of $A$ and $A'$ as $B$-freely independent subalgebras
 (see \cite{vo5} for an example of this construction).
We shall denote by $\rho \boxplus \nu := \mu_{x+y}$ as the \textit{additive B-free convolution} of $\rho$ and $\nu$ (here $x+y$ denotes the sum of these elements in this larger algebra).
  We say that an element $\mu \in \Sigma$
is \textit{infinitely divisible} with respect to $B$-valued additive free convolution if for every $n\in \mathbb{N}$ there exists an elements $\mu_{1/n} \in \Sigma$
such that $\mu = \mu_{1/n} \boxplus \cdots \boxplus \mu_{1/n}$, where the convolution on the right hand side is $n$-fold (this may be defined more generally
using free cumulants but this level of generality is enough for our present consideratinos, see \cite{Speichop}).

Let $M_{n}^{+}(B) = \{x\in M_{n}(B) : \exists \epsilon > 0 \ s.t. \ \Im{x}> \epsilon I_{n} \}$ where $I_{n}$ denotes the 
indentity element in $M_{n}(B)$.  
Let $H^{+}(B) = \bigsqcup_{n=1}^{\infty} M_{n}^{+}(B)$.  We shall refer to this sets as the
 \textit{noncommutative upper half-plane} ( $M_{n}^{-}(B)$ and the noncommutative lower half-plane are defined analagously).

We shall define a transform the encodes a distribution as a function on $\mathbb{H}^{+}(B)$.
Indeed, given a self-adjoint $a\in A$ with distribution $\mu$, we define the function
$G_{\mu}:\mathbb{H}^{+}(B) \rightarrow \mathbb{H}^{-}(B)$
by defining, for each $n$, a function $G_{\mu}^{(n)}: M_{n}(B)^{+} \rightarrow M_{n}(B)^{-}$
where $$G_{\mu}^{(n)}(b) = \mu((X\otimes I_{n}-b)^{-1})$$
Since $a$ is self adjoint, $a\otimes I_{n} -b$ is indeed invertible.  Consider the  series expansion
$$G_{\mu}^{(n)}(b)= b\sum_{n=0}^{\infty}\mu((b^{-1}X\otimes I_{n})^{n})$$
and observe that for $\mu$ with an exponential bound of $M$, we have that
this series is convergent for all $b$ such that $\|b^{-1} \| < M$.  We shall refer to the function $G_{\mu} $as the \textit{Cauchy transfrom}.
As is well known, the distribution $\mu$ may be recovered from its Cauchy transform.
We refer to \cite{vo5} for the noncommutative function theory associated to $B$-valued free probability.

We next define a map $F_{\mu}: \mathbb{H}^{+}(B) \rightarrow \mathbb{H}^{+}(B)$ by letting
$F^{(n)}_{\mu} = (G^{(n)}_{\mu})^{\langle -1 \rangle}$ where the superscript denotes the multiplicative inverse of this element. 
For each distribution $\mu \in \Sigma_{0}$, there exists a set $\Gamma_{n}$, which is a neighborhood of $\infty$  in $M_{n}(B)$ intersected with $M_{n}^{+}(B)$,
where $F^{(n)}_{\mu}$ is invertible.  For each $n$ we define the function $\varphi_{\mu}^{(n)}:M_{n}(B)^{+} \rightarrow M_{n}(B)^{-} $
by letting $\varphi_{\mu}^{(n)}(b) = (F_{\mu}^{(n)})^{-1}(b) - b$ (the superscript without the brackets refers
to the inverse with respect to composition).  We refer to the collection of all such maps
over $n$ as the \textit{Voiculescu transform} of $\mu$ (in symbols, $\varphi_{\mu}: H^{+}(B) \rightarrow H^{-}(B)$).  The fact that the image
of this map lies in the lower half plane is a consequence of the fact that
$\Im F_{\mu}^{(n)}(b) \geq \Im b$ for all $b\in M_{n}^{+}(B)$ (see \cite{PVB} for proof of this fact).  Given distributions
$\mu$ and $\nu$ we have that
$$\varphi^{(n)}_{\mu \boxplus \nu}(b) = \varphi^{(n)}_{\mu}(b) + \varphi^{(n)}_{\nu}(b)$$
for all $b\in M_{n}(B)^{+}$ in the common domain of the two right hand functions.

Let $\kappa_{\mu,n,k}: \bigotimes_{1}^{n} M_{k}(B) \rightarrow M_{k}(B)$ denote the \textit{cumulant functions}.  These functions
are studied extensively in \cite{Speichop} and satisfy the following properties:
\begin{enumerate}
 \item $\kappa_{\mu \boxplus \nu,n,k} (b_{1},\ldots,b_{n}) = \kappa_{\mu,n,k} (b_{1},\ldots,b_{n}) + \kappa_{\nu,n,k} (b_{1},\ldots,b_{n})$

\item $\varphi_{\mu}^{(k)}(b) = \sum_{n=1}^{\infty} \kappa_{\mu,n,k}(b^{-1} , \ldots , b^{-1})$

\item $\kappa_{\mu,n,k}(b_{1},\ldots,b_{n}) \leq M(4M)^{n}\|b_{1} \| \cdots \|b_{n} \|$ for $\mu \in \Sigma_{0,M}$
\end{enumerate}

\begin{remark}\label{domain_remark}
 Property (3) above implies that for $\mu \in \Sigma_{0,M}$, the power series in equation $(2)$ is convergent for all $b\in M_{k}(B)^{+}$ such
that $\|b^{-1} \| < 1/4M$.  This, as well as the fact that $G_{\mu}^{(k)}(b)$ may be written as a convergent series for all $b\in M_{k}(B)^{+}$ satisfying 
$\|b^{-1} \| < 1/M$, will be exploited in the following lemma.

\end{remark}

\begin{lemma}\label{convergence_lemma}
Let $\{ \mu_{\lambda} \}_{\lambda \in \Lambda} \in \Sigma_{0,M}$.  The following are equivalent.

\begin{enumerate}
 \item $\mu_{\lambda} \rightarrow \mu$ in the pointwise weak topology.

\item $G^{(k)}_{\mu_{\lambda}}(b) \rightarrow G^{(k)}_{\mu}(b)$ weakly for every $b\in M_{k}(B)^{+}$ such that $\|b^{-1} \| < 1/M$.

\item $F^{(k)}_{\mu_{\lambda}}(b) \rightarrow F^{(k)}_{\mu}(b)$ weakly for every $b\in M_{k}(B)^{+}$ such that $\|b^{-1} \| < 1/M$.

\item $\varphi_{\mu}^{(k)}$ and $\varphi^{(k)}_{\mu_{\lambda}}$ have analytic extension to the set of all $b\in M_{k}(B)^{+}$ such that $\|b^{-1} \| < 1/4M$.  
Furthermore, $\varphi_{\mu_{\lambda}}^{(k)}(b) \rightarrow \varphi_{\mu}^{(k)}(b)$ weakly for all $b$ in this set.
\end{enumerate}

\end{lemma}

\begin{proof}
The equivalence of $(2)$ and $(3)$ is obvious since the map sending an operator to its multiplicative inverse is continuous in this topology.

If we assume $(1)$, then parts $(2)$ and $(4)$ follow immediately.  Indeed, if we consider 
$$G_{\mu_{\lambda}}^{(k)}(b) = b\sum_{n=0}^{N}\mu_{\lambda}((b^{-1}X\otimes I_{k})^{n}) + b\sum_{n=N+1}^{\infty}\mu_{\lambda}((b^{-1}X\otimes I_{k})^{n})$$
for $N$ large enough, the second term on the right hand side of equation has norm less than $\epsilon$ (this may be done uniformly over $\lambda$ since all of the $\mu_{\lambda}$ are
elements of $\Sigma_{0,M}$).  Since the moments $\mu_{\lambda}((b^{-1}X\otimes I_{k})^{n})$ converge weakly to $\mu((b^{-1}X\otimes I_{k})^{n})$ uniformly for
$0 \leq n \leq N$, this implies that $G^{(k)}_{\mu_{\lambda}}(b) \rightarrow G^{(k)}_{\mu}(b)$ weakly.  As the above proof only relies on the fact that the transform
may be written as a convergent series in the moments of our distributions, $(4)$ follows in a similar manner.

Regarding $(2)$  $\Rightarrow$ $(1)$, recall that we may recover a distribution from either its Voiculescu or Cauchy transform when considered as a
fully matricial function (that is, consider $G_{\mu}^{(n)}$ for all $n\in \mathbb{N}$, see \cite{vo5}).
Compactness of $\Sigma_{0,M}$ implies that $\{\mu_{\lambda} \}_{\lambda \in \Lambda}$ has cluster points in this set.
  Assuming $(2)$, if $\nu$ is any cluster point, by the argument in
the previous paragraph, $G_{\nu}$ is equal to $G_{\mu}$ for those $b$ for which $\|b^{-1} \|< 1/M$.  Our claim follows from analytic continuation and the fact
that we may recover a distribution from its Cauchy transform.
The proof for $(4)$ $\Rightarrow$ $(1)$ is similar.
\end{proof}

\section{$W^{\ast}$-Algebras.}
Let $(A,\tau)$ be a tracial $W^{\ast}$-algebra and $B\subset A$ a $W^{\ast}$-subalgebra.  There is a natural
$B$-valued probability space $(A,E_{B},B)$ where $E_{B}:A \rightarrow B$ is the canonical conditional expectation (we shall refer to this triple as
a \textit{tracial $W^{\ast}$-probability space}).  We
refer to \cite{SS} for an introduction to these constructions.  We isolate the following facts for easy reference.

\begin{lemma}\label{expectations}
Let $(A,E_{B},B)$ be as above.  The expectation $E_{B}$ has the following properties:
\begin{enumerate}
 \item 
$E_{B}$ is a contraction with $E_{B}(1)=1$
\item
$E_{B}(bac) = bE_{B}(a)c$ for all $a\in A$ and $b,c\in A$.
\item
$\tau(E_{B}(x)y) = \tau(xE_{B}(y)) = \tau(E_{B}(x)E_{B}(y))$ for all $x,y \in A$
\item
$E_{B}$ is a normal, completely positive map.
\end{enumerate}
Moreover, $E_{B}$ is the unique trace preserving map that satisfies property $(2)$  
\end{lemma}

Let $\Sigma_{0}^{\tau} \subset \Sigma_{0}$ denote those elements $\mu$ so that $\mu(P(X)) = \tau(E_{B}(P(a)))$ where $(A,E_{B},B)$ form a tracial $W^{\ast}$-probability space.
  These distributions were studied extensively in \cite{vo5} and \cite{vo6}.  Note that this space is
closed under the $\boxplus$ operation through amalgamated free product constructions.  The primary purpose of this
section is to show that $\Sigma_{0}^{\tau} \cap \Sigma_{0,M}$ is compact in the pointwise weak toplogy (Corollary \ref{compact_corr}).  We begin our study of these distributions
with the following lemma.

\begin{lemma}\label{contraction_lemma}
Let $(A,E_{B},B)$ be a tracial $W^{\ast}$-probability space.  Assume that $X,Y \in A$ are $B$-free and 
$E_{B}(Y) = 0$.  Then $\| X \| \leq \|X + Y \|$.
\end{lemma}
\begin{proof}
Let $E_{B\langle X \rangle}: A \rightarrow B\langle X \rangle$ denote the canonical condition expectation.
Observe that 
$$\tau(E_{B\langle X \rangle}(Y)P(X)) = \tau(YP(X)) = \tau(YP(X)1) = \tau(YP(X)E_{B}(1))$$
$$ = \tau(E_{B}(YP(X))) = \tau(E_{B}(Y)E_{B}(P(X)))= 0$$
for all $P(X)\in B\langle X \rangle$.  The first two equalities follow from the properties of the expectation and the next to last equality follows 
from $B$-freeness.
This implies that $E_{B\langle X \rangle}(Y) = 0$ so that $E_{B\langle X \rangle}(X+Y) = X$.  As this map is a contraction,
we have that $\|X \| \leq \|X+Y \|$
\end{proof}

The following subordination result was originally proven in \cite{Bi}. A simple approach to this theorem
utilizing the structure of bialgebras was developed by Voiculescu \cite{vo4}.  We also refer to \cite{SC} for an extension of this
theorem to free compression semigroups.

\begin{theorem}\label{sub}
Let $X,Y \in A$ be $B$-free random variables.  Then, there exists a holomorphic map 
$\Phi^{(n)}:M_{n}^{+}(B) \rightarrow M_{n}^{+}(B)$ such that 
$$E_{M_{n}(B)\langle X \rangle} ([(X+Y)\otimes I_{n} - b]^{-1}) = (X \otimes I_{n} - \Phi^{(n)}(b))^{-1}$$
\end{theorem}

This theorem implies that, for distributions $\mu$ and $\nu$, the above holomorphic map satisfies
$$F^{(n)}_{\mu \boxplus \nu}(b) = F^{(n)}_{\mu}(\Phi^{(n)}(b)) $$ for $b\in M_{n}(B)^{+}$.  The following lemma is
a simple consequence of subordination and is a slightly more general version of remark \ref{domain_remark} for this class of distributions.
  We follow \cite{Wil2} which addresses the scalar-valued case.  For $\mu,\nu, \rho \in \Sigma_{0}^{\tau}$ satisfying $\mu = \nu \boxplus \rho$, we shall 
refer to the distributions $\nu$ and $\rho$ as \textit{factors} of $\mu$.
\begin{lemma}
Given a $B$-valued distribution $\mu$ there exists an open set $\Gamma \subset M_{n}^{+}(B)$ such that
$(F_{\nu}^{(n)})^{-1}$ has analytic continuation to $\Gamma$ for all factors $\nu$ or $\mu$.  Moreover, 
$\Im(\varphi_{\nu}^{(n)})(b) \leq 0$ for all $b\in \Gamma$.
\end{lemma}
\begin{proof}
Fix $n\in \mathbb{N}$.  In what follows, we will drop the $n$ and refer to the functions of the form $F_{\rho}^{(n)}$ as $F_{\rho}$.
As is well know (see \cite{vo5}), there exists subsets $\Gamma_{1}$ and $\Gamma_{2}$ of the form $\{b \in M_{n}(B)^{+}: \Im{b} > \alpha , \ \Im{b} > \beta \Re{b} \}$
so that $F_{\mu}^{-1}$ and $F_{\nu}^{-1}$ are respectively defined and have positive imaginary part.  
Utilizing Theorem \ref{sub}, we have that $\Phi^{(n)} \circ F_{\mu}^{-1}(b) = F_{\nu}^{-1}(b)$ for $b\in \Gamma_{1} \cap \Gamma_{2}$.  Since
the left hand side may be continued to $\Gamma_{1}$, the same must be true of $F_{\nu}^{-1}$.

With respect to the negativity claim, observe that $b = F_{\nu}(F_{\nu}^{-1}(b)) = F_{\nu}(\Phi^{(n)} \circ F_{\mu}^{-1}(b))$ for $b\in \Gamma_{1} \cap \Gamma_{2}$
and that, through continuation, this is true for $b\in \Gamma_{1}$.  Recall that $F_{\nu}$ satisfies $\Im F_{\nu}(b)\geq \Im b$.  Thus, abusing notation by letting
$\varphi_{\nu}^{(n)}$ denote the extension of the Voiculescu transfrom to $\Gamma_{1}$, we have the following:
$$\varphi_{\nu}^{(n)}(b) = \Phi^{(n)} \circ F_{\mu}^{-1}(b) - b = \Phi^{(n)} \circ F_{\mu}^{-1}(b) - F_{\nu}(\Phi^{(n)} \circ F_{\mu}^{-1}(b))$$
and our claim follows.
\end{proof}

We close the section with a theorem providing necessary and sufficient conditions that a distribution arises from a conditional expectation
of tracial von Neumann algebras.

\begin{theorem}
Let $\mu \in \Sigma_{0}$ where $B$ is assumed to be a tracial von Neumann algebra.  
Then, $\mu \in \Sigma_{0}^{\tau}$ if and only if the following conditions hold for all $P(X), Q(X) \in B\langle X \rangle$:
\begin{enumerate}
 \item $\tau(\mu(P^{\ast}(X)XP(X))) \leq M \tau(\mu(P^{\ast}(X)P(X)))$ .

\item  $\tau(\mu(P^{\ast}(X)Q(X)))^{2} \leq \tau(\mu(P^{\ast}(X)P(X))) \tau(\mu(Q^{\ast}(X)Q(X)))  $

\item  $\tau(\mu(P(X)Q(X))) = \tau(\mu(Q(X)P(X)))$
\end{enumerate}

\end{theorem}

\begin{proof}
Assume that $\mu(P(X)) = E_{B}(P(a))$.  Condition $(1)$ follows from positivity of $E_{B}$.  Condition $(2)$ follows from the fact that $E_{B}$ preserves $\tau$ and
the Cauchy-Schwarz inequality.  Condition $(3)$ follows from the fact that $E_{B}$ preserves $\tau$ and that $\tau$ is a trace.

We next assume conditions $(1)$, $(2)$ and $(3)$ above.  Let $\mathcal{N} = \{N(X) \in B\langle X \rangle : \tau(\mu(N(X)^{\ast}N(X))) = 0 \}$.
Condition $(3)$ implies that $\mathcal{N}$ is closed under the adjoint operation.  Conditions $(2)$ and $(3)$ implies that $\tau(N(X)P(X)) = \tau(P(X)N(X)) = 0$ for all
$P(X) \in B\langle X \rangle$.  Thus, $\mathcal{N}$ is a $2$-sided ideal in $B\langle X \rangle$.

We consider $B\langle X \rangle / \mathcal{N}$.  Observe that $\tau(\mu((Q^{\ast}(X) + N(X))(P(X) + N'(X)))) = \tau(\mu(Q^{\ast}(X)P(X))) $ for all
$P(X),Q(X) \in B\langle X \rangle$ and $N(X) , N'(X) \in \mathcal{N}$.  Thus, we have a well defined inner product 
$\langle P(X) + \mathcal{N}, Q(X) + \mathcal{N} \rangle = \tau(\mu(Q^{\ast}(X) P(X)))$ so that $B\langle X \rangle / \mathcal{N}$ is a pre-Hilbert space.
We denote by $\mathcal{H}$ its completion with respect to the norm defined by this inner product (in symbols, $\| \cdot \|_{\mathcal{H}}$ is the inner product norm
and $\| \cdot \|$ is the norm on $B$).

We define an action of $B\langle X \rangle / \mathcal{N}$ on $\mathcal{H}$ through left multiplication.  Observe that, for $b\in B$, we have
that $\|b \|^{2} - b^{\ast}b$ is a positive element in $B$.  Thus, $\|b \|^{2} - b^{\ast}b = c^{\ast}c$ for some $c\in B$.  As our notion of positivity of $\mu$
is purely algebraic, we have that $\mu(P^{\ast}(X)c^{\ast}cP(X)) \geq 0$ so that $\|b\|^{2} \mu(P^{\ast}(X)P(X)) \geq \mu(P^{\ast}(X)b^{\ast}bP(X))$ for all $b\in B$
and $P(X) \in B\langle X \rangle$.  Therefore, given a monomial $Xb_{1}X \cdots Xb_{n}X \in B\langle X \rangle$, we have the following:

\begin{align*}
\| ( (Xb_{1}X \cdots Xb_{n}X + \mathcal{N}) & \cdot (P(X) + \mathcal{N}) \|_{\mathcal{H}}^{2} \\
    =\tau(\mu(P^{\ast}(X) Xb_{n}^{\ast}X \cdots Xb_{1}^{\ast} & XX b_{1}X \cdots Xb_{n}X P(X))  \\
  \leq 
M^{2} \tau(\mu(P^{\ast}(X) Xb_{n}^{\ast}X \cdots  X & b_{1}^{\ast} b_{1}X \cdots Xb_{n}X P(X))  \\
 \leq 
M^{2} \|b_{1} \|^{2} \tau(\mu(P^{\ast}(X) Xb_{n}^{\ast}X & \cdots b_{2}^{\ast}XX b_{2}\cdots Xb_{n}X P(X))  \\
\end{align*}

By induction, we have that $$\| ( (Xb_{1}X \cdots Xb_{n}X + \mathcal{N}) \cdot (P(X) + \mathcal{N}) \|_{\mathcal{H}} \leq M^{n+1} \|b_{1} \| \cdots \|b_{n} \| \|P(X) + \mathcal{N} \|_{\mathcal{H}}
$$  As this holds for all $P(X) + \mathcal{N} \in B\langle X \rangle / \mathcal{N}$ which is dense in $\mathcal{H}$, we have that the monomials are bounded operators
on this Hilbert space.  Extending through linearity, we may imbed $B\langle X \rangle / \mathcal{N}$ into $B(\mathcal{H})$.  Let $A$ denote the weak
closure of its image.

The map $\mu: B\langle X \rangle / \mathcal{N} \rightarrow B$ is well defined since, for $N(X) \in \mathcal{N}$ and $b \in B$,
we have that $\tau(\mu(N(X))b) = \tau(\mu(N(X)b)) = 0$ by $(2)$.   This implies that $\mu(N(X)) = 0$.
To complete our proof, we must extend $\mu$ to all of $A$ and show that this extension is positive, faithful, B-bimodular and satisfies condition $(3)$.

First, for each $b \in B$ define $\xi_{b} = b + \mathcal{N} \in \mathcal{H}$.  Let $\{P_{\lambda}(X) + \mathcal{N} \}_{\lambda \in \Lambda}$ form a weakly Cauchy net in $A$.
  This implies that $\langle (P_{\lambda}(X) + \mathcal{N}) \xi_{1} , \xi_{(b'b)^{\ast}}\rangle = \tau(\mu(P_{\lambda}(X))bb' )$ is 
Cauchy in $\mathbb{C}$ for all $b,b' \in B$.  Since functionals of this type induce the weak operator toplogy on B (with respect to the standard represenation),
we have that the set $\{\mu(P_{\lambda}(X) + \mathcal{N}) \}_{\lambda \in \Lambda}$ is Cauchy in the weak operator topology so that we have a well
defined extension with $\mu(a) = \lim_{\lambda}(\mu(P_{\lambda}(X) + \mathcal{N}))$ in this topology.

In order to prove positivity, we may further assume that the net $\{P_{\lambda}(X) + \mathcal{N}\}_{\lambda \in \Lambda}$ converges to $a \in A$ in the
strong operator topology.  Since products are continuous in this topology and the positive cone is weakly closed, we have that
$\mu(a^{\ast}a) = \lim_{\lambda}\mu((P^{\ast}_{\lambda}(X)P_{\lambda}(X)) \geq 0$.  To prove that our extension is faithful, we again assume that $P_{\lambda}(X) + \mathcal{N} \rightarrow a$
in the strong operator toplogy on $A$.  Assuming that $\mu(a^{\ast}a) = \lim_{\lambda}\mu(P^{\ast}_{\lambda}(X)P_{\lambda}(X)) = 0$ where the limit is in the weak operator
topology on $B$, we have that, for $Q(X), R(X) \in B\langle X \rangle$,
$$ \langle (P_{\lambda(X)} + \mathcal{N})\cdot (Q(X) + \mathcal{N}),(R(X) + \mathcal{N}) \rangle = \tau(\mu(P_{\lambda}(X)Q(X)R^{\ast}(X)))$$
and the right hand side goes to $0$ by condition $(2)$ and weak continuity of $\tau$.
As elements of this type are dense in $\mathcal{H}$, this implies that $a = 0$.  Bimodularity and condition $(3)$ follow through similar methods.

To finish the proof, we define a trace on $A$ by letting $\tau'(a) = \tau(\mu(a))$.  Note that $\tau'(\mu(a)) = \tau(\mu(\mu(a))) = \tau(\mu(1) \mu(a)) = \tau'(a)$
 so that $\mu$ is trace preserving.
$B$-bimodularity of $\mu$ implies, by Theorem \ref{expectations}, that $\mu$ is the canonical conditional expectation.
Lastly observe that $\mu(P(X) + \mathcal{N}) = \mu(P(X))$ so that our distribution arises from this expectation.

\end{proof}

\begin{corollary}\label{compact_corr}
 The set $\mu_{0}^{\tau}$ is closed in the topology of pointwise weak convergence.  In particular, $\Sigma_{0}^{\tau} \cap \Sigma_{0,M}$ is compact for
all $M \in \mathbb{R}^{+}$.
\end{corollary}

\begin{proof}
 Observe that, since $\tau$ is weakly continuous, conditions $(1)$, $(2)$, and $(3)$ are closed under pointwise weak limits.
Thus, $\Sigma_{0}^{\tau} \cap \Sigma_{0,M}$ is a closed subset of a compact set.
\end{proof}

\section{The Steinitz Lemma.}
The following theorem was originally proven by Steinitz in \cite{steinitz3}.
\begin{lemma}
Let $\{v_{i} \}_{i=1}^{k} \subset \mathbb{R}^{N}$ be a set of elements in the unit ball, where $\mathbb{R}^{N}$ is equipped
with the Euclidean metric.  Then, there exists a permutation $\sigma$ of $\{1,\ldots, k \}$ such that
for all $1\leq j \leq k$ we have that $|\sum_{i=1}^{j} v_{\sigma(i)}| \leq N$.
\end{lemma}
We refer to \cite{steinitz2} for a modern proof of this lemma.  We refer to \cite{steinitz} for a survey of its history
and applications to convex geometry.  The following simple corollary of this fact is singled out for easy reference.

\begin{corollary}\label{ST_corollary}
Consider vectors $\{v_{i} \}_{i=1}^{n} \subset \mathbb{R}^{N}$ such that $|v_{i}|\leq \epsilon$ and
$\sum_{j=1}^{n}v_{i} = v$. Then, for each $t\in (0,1)$, there exists a subset $\sigma \subset \{1,\ldots,n \}$ such that
$| \sum_{i \in \sigma} v_{i} - tv | \leq N \epsilon$.
\end{corollary}
\begin{proof}
We assume that $v=(|v|,0,\ldots,0)$, $v_{i} = (t_{i1},t_{i2},\ldots,t_{iN})$ 
and $w_{i} = (t_{i2},t_{i3}, \ldots, t_{iN}) \in \mathbb{R}^{N-1}$.  Observe that $\sum_{i=1}^{n}w_{i} = 0$ and $|w_{i}|\leq \epsilon$.
By the Steinitz lemma, we may assume that $|\sum_{i=1}^{\ell}w_{i}|\leq (N-1)\epsilon$ for all $\ell =1,\ldots,n$.
This implies that $\sum_{i=1}^{\ell}v_{i}$ is contained in a tube about the line passing through $v$ and the origin
of radius $(N-1)\epsilon$ for all $\ell = 1,\ldots,n$.  Since each of the $|v_{i}|$ has magnitude bounded by $\epsilon$,
the intermediate value property implies that there exists an $m \in \{1,\ldots,n \}$ such that 
$|\sum_{i=1}^{m}t_{i1} - t|v|| \leq \epsilon/2$.  For this $m$, we have that $|\sum_{i=1}^{m} v_{i} - tv| \leq N\epsilon$,
proving our result.
\end{proof}

The following is easily derived from Corollary \ref{ST_corollary}.  The details are left to the reader.

\begin{corollary}\label{ST_euclidean}
Let $t\in (0,1)$ and $\{v_{ij} \}_{i\in\mathbb{N}, j=1,\ldots,n_{i}} \subset \mathbb{R}^{N}$ satisfy
$\|v_{ij}\| \rightarrow 0$ uniformly over $j$ as $i\uparrow \infty$ and $\|(v_{i1} + \cdots + v_{in_{i}}) - v \| \rightarrow 0$ for some $v\in \mathbb{R}^{N}$.
Then, there exists a sequence of subsets $\sigma_{i} \subset \{1,2,\ldots,n_{i} \}$ such that $\|\sum_{j\in \sigma_{i}} v_{ij} - tv \| \rightarrow 0$ as $i\uparrow \infty$.
\end{corollary}

\begin{remark}\label{ST_hilbert}
Note that, through a trivial approximation argument, we may replace Euclidean space in the preceding corollary with Hilbert space.
\end{remark}

\section{Main Results}

We now formulate and prove our main result.  We assume throughout that $B \subset B(\mathcal{H})$ with $\mathcal{H}$ separable.
For an elements $b\in B$ we denote by $\delta_{b}$ the distribution defined by the equation $\delta_{b}(P(X)) = P(b)$ for all $P(X) \in B\langle X \rangle$.

\begin{theorem}\label{mainresult}
Consider  $\mu, \{\mu_{ij} \}_{i\in \mathbb{N}, j=1,\ldots,n_{i}} \subset \Sigma_{0}^{\tau}$ and self adjoint elements $\{b_{i} \}_{i\in \mathbb{N}} \subset B$
satisfying the following properties:

\begin{enumerate}
 \item $\mu_{i} = \mu_{i1} \boxplus \mu_{i2} \boxplus \cdots \mu_{in_{i}} \boxplus \delta_{b_{i}} \in \Sigma_{0,M}$ for all $i\in \mathbb{N}$.
\item $\mu_{i} \rightarrow \mu$ in the pointwise weak topology.
\item $\mu_{ij} \rightarrow \delta_{0}$ in the pointwise weak topology, uniformly over $j=1,\ldots,n_{i}$.
\end{enumerate}
Then, for each $n \in \mathbb{N}$, there is a $\mu_{1/n} \in \Sigma_{0}^{\tau}$ such that $\mu = \mu_{1/n} \boxplus \cdots \boxplus \mu_{1/n}$ where the convolution
on the right hand side is $n$-fold.
\end{theorem}

\begin{proof}
Fix $t = 1/p$ for $p \in \mathbb{N}$.  Let $\{\xi_{k} \}_{k\in \mathbb{N}} \subset \mathcal{H}$ denote a separable basis.  
We assume without loss of generality that each of the $\mu_{ij}$ satisfies $\mu_{ij}(X) = 0$.  Observe that lemma \ref{contraction_lemma} implies
that $\{b_{i} \}_{i=1}^{\infty}$ are bounded in norm so that, by \ref{compact_corr}, we may assume that $\{\delta_{b_{i}} \}_{i \in \mathbb{N}}$ 
converges in the pointwise weak topology (note that the limit point is an element of $\Sigma_{0}^{\tau}$ but need not be of the form $\delta_{b}$ for $b\in B$).
Also observe that, by lemma \ref{contraction_lemma}, $\boxplus_{j\in \sigma}\mu_{ij} \boxplus \delta_{b_{i}} \in \Sigma_{0,M}$ for any subset $\sigma \subset \{1,2,\ldots,n_{i} \}$, 
so that $\boxplus_{j\in \sigma} \mu_{ij} \in \Sigma_{0,2M}$.  

Now, for a Voiculescu transfrom $\varphi_{\nu} = (\varphi_{\nu}^{(\ell)})_{\ell \in \mathbb{N}}$, we restrict
our attention to $\varphi_{\nu}^{(1)}$.  Let $\{d_{n} \}_{n\in \mathbb{N}} \subset B$ be a family of self adjoint elements
with dense linear span.  Consider $c_{n}= d_{n} + i\lambda I$ where $I$ is the unit in $B$ and $\lambda > 16M$ so that $\|c_{n}^{-1} \| < 1/16M$.
Note that $\varphi^{(1)}_{\mu}$ and $\varphi^{(1)}_{\boxplus_{j\in \sigma} \mu_{ij}}$ are defined on $\{c_{n} \}_{n\in \mathbb{N}}$ for
 all $i\in \mathbb{N}$ and $\sigma \subset \{1,\ldots,n_{i} \}$, and that this function is completely determined by its values on this countable set.

The idea of the proof is to use the Steinitz lemma to construct a sequence of decompositions $\mu_{i} = \nu_{i} \boxplus \rho_{i}$ so
that $\nu_{i}$ subconverges to $\mu_{t}$.  Since this lemma is for finite dimensional spaces, we must truncate the Voiculescu transform.
Towards this end, let $P_{M}: \mathcal{H} \rightarrow \mathbb{R}^{2M}$ be defined by
$P_{M}(\sum_{k=1}^{\infty} \alpha_{k} \xi_{k}) = (\Im \alpha_{1}, \Re \alpha_{1}, \ldots , \Im \alpha_{M} , \Re \alpha_{M})$.  We then define a map
$\Phi_{K,M,N}: \Sigma_{0,M} \rightarrow \mathbb{R}^{2KMN}$ as follows:
$$\Phi_{KMN}(\mu):= \prod_{k=1}^{K} \prod_{n=1}^{N} P_{M}(\varphi^{(1)}_{\mu}(c_{n}) \cdot \xi_{k}) $$  The purpose of this construction is that the relevant
transforms are completely determined by their values on $\{ c_{n}\}_{n\in \mathbb{N}}$.  Since $\phi_{\mu}^{(1)}(c_{n}) \in B$, these elements are completely determined
by their action on this basis for $\mathcal{H}$. Thus, $\varphi_{\mu}^{(1)}$ may be recovered from $\{\Phi_{KMN}(\mu) \}_{K,M,N \in \mathbb{N}}$.

Observe that $$\sum_{j=1}^{n_{i}}\Phi_{KMN}(\mu_{ij}) = \Phi_{KMN}(\mu_{i1} \boxplus \cdots \boxplus \mu_{in_{i}}) \rightarrow \Phi_{KMN}(\mu) - \Phi_{KMN}(\delta_{b_{i}})$$
Further note that the assumption that $\mu_{ij}\rightarrow \delta_{0}$ in the pointwise weak topology, as well as lemma \ref{convergence_lemma}, implies that
$\Phi_{KMN}(\mu_{ij}) \rightarrow (0,\ldots,0)$ uniformly over $i$.  If $v$ is the limit point of $ \Phi_{KMN}(\mu) - \Phi_{KMN}(\delta_{b_{i}})$, by corollary
\ref{ST_euclidean}, there exists a sequence of subsets $\sigma_{i} \subset \{ 1,2,\ldots,n_{i} \}$ such that $\sum_{j\in \sigma_{i}} \Phi_{KMN}(\mu_{ij}) \rightarrow tv$.
Thus, up to truncation, any cluster point of the sequence of measures $\{\boxplus_{j\in \sigma_{i}} \mu_{ij} \boxplus \delta_{tb_{i}} \}_{i \in \mathbb{N}}$ will have
Voiculescu transform equal to $\varphi_{\mu_{1/p}}$.

Observe that the above proof also works for $\varphi_{\mu}^{(\ell)}$ for all $\ell \in \mathbb{N}$.  Thus, if we diagnolize over $\ell, K, M$ and $N$, we obtain a
sequence of subsets $\sigma_{i} \subset \{1,\ldots,n_{i} \}$ such that $\varphi_{\boxplus_{j\in\sigma_{i}}\mu_{ij} \boxplus \delta_{tb_{i}}}  \rightarrow t\varphi^{\ell}_{\mu}$ in 
the pointwise weak topology on an open set $\Gamma \subset H^{+}$.  As we saw in the opening comments, the sequence $\boxplus_{j\in\sigma_{i}}\mu_{ij} \boxplus \delta_{tb_{i}}$ is
 contained in $\Sigma_{0,2M} \cap \Sigma_{0}^{\tau}$.
By Corollary \ref{compact_corr}, this set is compact in the pointwise weak topology.  By lemma \ref{convergence_lemma}, any cluster point of this sequence 
will have the required distribution, so our theorem holds.
\end{proof}

\section{Conculusion and Acknowledgements.}\label{conclusion}

We begin by noting that, while the above proof may seem quite complex due to the correspondingly complex machinery, the underlying idea
is quite simple.  Indeed, if a measure $\mu$ is the limit of an infinitesimal array $\{ \mu_{ij} \}_{i\in \mathbb{N}, j=1,\ldots, n_{i}}$,
then, after taking the appropriate transforms and utilizing the Steinitz lemma, we  we may construct a sequence of subsets
$\sigma_{i} \subset \{1,2,\ldots,n_{i} \}$ so that $\boxplus_{j\in \sigma_{i}} \mu_{ij}$ converges to $\mu_{1/n}$.  In $\mathbb{R}^{N}$, this is precisely
 the Steinitz lemma, so that the whole approach is to come up with a family of maps into $\mathbb{R}^{n}$
that allow us to exploit this lemma.

Observe that the proof of our main result may be adapted to other probabilistic settings.  Indeed, if we
are to consider the classical theorem due to Hin\u{c}in, the above proof may be adapted with the logarithm
of the Fourier transform replacing the Voiculescu transform.  Furthermore, utilizing remark \ref{ST_hilbert},
one would expect this proof to work for vector valued probability distributions.  This is a somewhat more intuitive
construction than the tradtional function theoretic approach to these theorems since the more classical approach does
not include the observation that a subset of the infinitesimal array actually converges to the distribution $\mu_{1/n}$.

Lastly, this project has raised questions about the suitability of various weak topologies to the theory.  In this paper, we develop
a theory of pointwise weak convergence which, although slightly unnatural in an operator algebra,
has the desirable properties that the unit ball is compact and that convergence in this topology corresponds to convergence of our various transforms.
  However, it is unclear whether weak topologies that are more intrinsic to these operator algebras behave well
with respect to the transformations. In particular, the question arises as to whether weak convergence of elements in an operator algebra $A$ corresponds
to some type of convergence for their Voiculescu transforms. In the scalar valued case, weak convergence is equivalent to uniform convergence of the Voiculescu
transforms on certain compact subsets in the complex upper half space (see \cite{BV1}).  Is there a corresponding theorem in the more general operator valued case?

\begin{acknowledgements}
I would like to thank Hari Bercovici, Michael Anshelevitch, Michael Hartglass and Ken Dykema for their helpful advice.  I would also like
to thank Imre B\'{a}r\'{a}ny for referring me to the relevant literature on the Steinitz lemma.
\end{acknowledgements}

\bibliographystyle{abbrv}
\bibliography{Levy_Hincin_Scalar_Valued}

\end{document}